\documentclass[11pt]{article}
\usepackage{mathrsfs}
\usepackage{amssymb,amsfonts,amsmath,color,amsthm,url}
\usepackage[margin=1in]{geometry}
\newtheorem{thm}{Theorem}[section]
\newtheorem{lemma}[thm]{Lemma}

\newtheorem{prop}[thm]{Proposition}
\newtheorem{conj}[thm]{Conjecture}

\usepackage{titlesec}

\titleformat{\subsection}[runin]
  {\normalfont\bfseries}
  {\thesubsection.}
  {0.5em}
  {}[.]
\begin{document}
\title{A spectral condition for perfect matchings in 3-partite 3-graphs\thanks{Supported by  National Key Research and Development Program of China 2023YFA1010203}}
\author{ Hongliang Lu and Feihong Yuan
\medskip\\ {\small School  of Mathematics and Statistics}\\ {\small Xi'an Jiaotong University, Xi'an,   China}}

\date{}
\maketitle
\begin{abstract}
Let \(H\) be a 3-partite 3-uniform hypergraph whose three vertex classes all have size \(n\). For a vertex $v \in V(H)$, the link graph $N_H(v)$ is defined on $V(H)\setminus\{v\}$ with edge set $\{e\setminus\{v\}: v\in e\in E(H)\}$, and we denote by $\rho(N_H(v))$ its spectral radius. We prove that for every $\alpha>0$ there exists $n_0$ such that for all $n\ge n_0$ the following holds: if  
\[
\rho\bigl(N_H(v)\bigr) > \left(\frac{\sqrt{2}}{2}+\alpha\right)n
\]  
for every vertex $v\in V(H)$, then $H$ contains a perfect matching. This spectral condition is asymptotically best possible. 
\end{abstract}

\section{Introduction}
Let $k\ge 2$ be an integer.  A $k$-uniform hypergraph, or simply a
$k$-graph, is a pair $H=(V,E)$, where $V$ is a finite set and
$E\subseteq \binom{V}{k}$. A \emph{matching} in $H$ is a collection of
pairwise vertex-disjoint edges.  We write $\nu(H)$ for the maximum size of a
matching in $H$.  A matching is \emph{perfect} if it covers every vertex of
$H$. For a real-valued function $h$ on a finite set $X$ and any subset $S \subseteq X$, we write $h(S) := \sum_{x \in S} h(x)$. 

For $S\subseteq V(H)$, let $d_H(S)$ denote the number of edges of $H$
containing $S$.  If $0\le \ell<k$, the minimum $\ell$-degree of $H$ is
\[
        \delta_\ell(H)
        :=
        \min\left\{d_H(S): S\in \binom{V(H)}{\ell}\right\}.
\]
In particular, $\delta_0(H)=e(H)$ and $\delta_1(H)$ is the minimum vertex
degree.  For an $\ell$-set $S\subseteq V(H)$, define
\[
        N_H(S)
        :=
        \left\{U\in \binom{V(H)\setminus S}{k-\ell}: S\cup U\in E(H)\right\}.
\]
Thus $d_H(S)=|N_H(S)|$.  When $S=\{v\}$, we write $N_H(v)$ and $d_H(v)$
instead of $N_H(\{v\})$ and $d_H(\{v\})$.  The set $N_H(v)$ will also be
viewed as the $(k-1)$-uniform link hypergraph of $v$, with vertex set
$V(H)\setminus \{v\}$.  For $1\le \ell<k$, let $m_\ell(k,n)$ be the smallest
integer $m$ such that every $k$-graph $H$ on $n$ vertices with
$\delta_\ell(H)\ge m$ contains a perfect matching.

For $\ell\ge 1$, the thresholds $m_\ell(k,n)$ are expected to be determined
asymptotically by the maximum of the space barrier and the divisibility
barrier.  H{\`a}n, Person and Schacht~\cite{HPS} conjectured that, for every
$1\le \ell<k$,
\begin{align}\label{conj1}
        m_\ell(k,n)
        \sim
        \max\left\{
        \frac12,\,
        1-\left(1-\frac1k\right)^{k-\ell}
        \right\}
        \binom{n}{k-\ell}.
\end{align}
The codegree case $\ell=k-1$ was determined exactly by R\"odl, Ruci\'nski and
Szemer\'edi~\cite{rodl2009}, improving an earlier bound from~\cite{rodl2006}.
For $k=3$ and $\ell=1$, H{\`a}n, Person and Schacht~\cite{HPS} proved 
$m_1(3,n)\sim \frac59\binom{n}{2}$,
and the exact threshold was later obtained independently by
Khan~\cite{Kh13} and by Kühn, Osthus and Treglown~\cite{KOT13}.  Khan
\cite{kh2016} further determined $m_1(4,n)$ exactly, and Frankl, Lu, Ma and
Wu~\cite{FLWW2026} recently determined $m_1(5,n)$ exactly.  For
$k/2\le \ell<k$, Pikhurko~\cite{pikhurko2008} proved $ m_\ell(k,n)\sim \frac12\binom{n}{k-\ell}$,
with exact versions later obtained by Czygrinow and Kamat~\cite{cz2012} and
by Treglown and Zhao~\cite{treg2013,treg2016}.  Other asymptotic cases of
\eqref{conj1} were established by Alon et al.~\cite{AFH12} for
$k-\ell\le 4$, by Han~\cite{han2016} for $0.42k<\ell<0.5k$, and by Lu and
Yu~\cite{luyu2022} for $0.40k<\ell<0.42k$.  Additional results in this
direction can be found in \cite{kuhn2006,mark2011,treg2016}.

A closely related setting is that of partite hypergraphs.  A $k$-graph $H$
is called \emph{$k$-partite} if its vertex set admits a partition
\[
        V(H)=V_1\cup \cdots\cup V_k
\]
such that every edge contains exactly one vertex from each class.  If
$|V_1|=\cdots=|V_k|=n$, then $H$ is called an \emph{$n$-balanced
$k$-partite $k$-graph}. In the $3$-partite $3$-uniform case, Lo and Markström~\cite{lo2014}
determined the minimum degree threshold for the existence of a perfect
matching.  Further results on perfect matchings in balanced partite
hypergraphs include \cite{Aharoni2009,lwy2019,pikhurko2008}.

This paper studies a spectral variant of the partite problem.  For a graph
$G$, let $A(G)$ be its adjacency matrix, and let $\rho(G)$ denote the largest
eigenvalue of $A(G)$, called the \emph{spectral radius} of $G$.  Spectral
conditions for perfect matchings in hypergraphs have recently been considered
from the viewpoint of vertex adjacency graphs.  Lin, Lu, Yuan
and Zhao~\cite{LLYZ2026} obtained asymptotically tight spectral conditions
for perfect matchings in $3$-graphs.
In this paper, we prove that a  spectral condition on the link graphs of a 3-partite 3-uniform hypergraph ensures the existence of a perfect matching. 
\begin{thm}\label{main}
For every $\gamma>0$, there exists an integer $n_0$ such that the following  holds for all $n\geq n_0$. 
     Let $H$ be an $n$-balanced 3-partite 3-graph. If $\rho(N_H(v))>(\sqrt{2}/2+\gamma)n$ for every $v\in V(H)$, then $H$ contains a perfect matching.
\end{thm}

The bound in Theorem~\ref{main} is asymptotically tight. To see this, let $V_1,V_2,V_3$ be the three pairwise disjoint sets of size
$n$. Choose subsets $A_1\subseteq V_1$ and $A_2\subseteq V_2$ such that  
\[
|A_1| = |A_2| = \left\lfloor \frac{n-1}{2} \right\rfloor.
\]
When $n$ is odd, define a $3$-partite $3$-graph $H_{\mathrm{odd}}$ on
$V_1\cup V_2\cup V_3$ by
\[
        E(H_{\mathrm{odd}})
        =
        \bigl\{\{x,y,z\}\in \prod_{i\in [3]}V_i: 
        x\in A_1 \text{ or } y\in A_2\bigr\}.
\]
Every matching in $H_{\mathrm{odd}}$ has size at most
$|A_1|+|A_2|=n-1$, and hence $H_{\mathrm{odd}}$ has no perfect matching.
Moreover,
\[
        \min_{v\in V(H_{\mathrm{odd}})}
        \rho(N_{H_{\mathrm{odd}}}(v))
        =
        \sqrt{\frac{n(n-1)}{2}}.
\]
For even $n$, fix a vertex $u\in V_3$ and define $H_{\mathrm{even}}$ by
\[
        E(H_{\mathrm{even}})
        =
        \bigl\{\{x,y,z\}: 
        x\in A_1 \text{ or } y\in A_2 \text{ or } z=u\bigr\}.
\]
Again $H_{\mathrm{even}}$ has no perfect matching.  A direct computation of
the spectral radii of the link graphs gives
\[
\min_{v\in V(H_{\mathrm{even}})}
\rho(N_{H_{\mathrm{even}}}(v))
=
\frac12
\sqrt{
n^2-n+2+
\sqrt{n^4-6n^3+9n^2+12n-12}
}.
\]
These constructions suggest the following exact form of the threshold.
\begin{conj}
Let $H$ be a $3$-partite $3$-graph with vertex classes of size $n$.  Define
\[
\tau(n)=
\begin{cases}
\displaystyle
\sqrt{\frac{n(n-1)}{2}},
& \text{if $n$ is odd},\\[12pt]
\displaystyle
\frac12
\sqrt{
n^2-n+2+
\sqrt{n^4-6n^3+9n^2+12n-12}
},
& \text{if $n$ is even}.
\end{cases}
\]
If
\[
        \rho(N_H(v))>\tau(n)
\]
for every vertex $v\in V(H)$, then $H$ contains a perfect matching.
\end{conj}

The remainder of the paper is organized as follows.  Section~2 collects the
notation and preliminary tools used throughout the proof.  Section~3 proves
the fractional matching statement that serves as the main technical input.
Section~4 completes the proof of Theorem~\ref{main}.

\section{Preliminaries}
In this section we collect the notation and auxiliary results used in the
proof. The statement ‘\(a \ll b\)’ means that there exists a non‑increasing function \(f\) such that whenever \(a < f(b)\) the subsequent claims hold.  We  record a standard spectral estimate for bipartite graphs.

\begin{lemma}[Bhattacharya, Friedland and Peled, \cite{bh2008}]\label{lem:bipartite-spectral}
Let $G$ be a bipartite graph with $m$ edges. Then
\[
\rho(G)\le \sqrt{m}.
\]
\end{lemma}

We shall also use the following absorption lemma of Lo and Markstr\"om~\cite{lo2014}.
For a $k$-partite $k$-graph with vertex classes $V_1,\dots,V_k$, a set
$S\subseteq V(H)$ is called \emph{legal} if $|S\cap V_i|\le 1$ for every
$i\in[k]$.  The minimum $\ell$-degree $\delta_{\ell}(H)$ of a $k$-partite $k$-graph is taken
over all legal $\ell$-sets.
\begin{lemma}[Lo and Markstr\"om,~\cite{lo2014}]\label{absorb}
Let  $1 \leq \ell<k$, $0<\gamma< 1 /\left(10 k^{3}\right)$ and $\gamma^{\prime}=\gamma^{2 k-1} / 20$.
Then there is an integer $n_{0}$ such that for all $n>n_{0}$ the following holds:
Suppose $H$ is a $k$-partite $k$-graph with $n$ vertices in each class and minimum $\ell$-degree
$\delta_{\ell}(H) \geq(1 / 2+\gamma) n^{k-\ell}$.
Then there exists a matching $M$ in $H$ of size $|M| \leq(k-1) \gamma^{k} n$ such that, for every balanced set $W$ of size $|W| \leq k \gamma^{\prime} n$, there exists a matching covering exactly the vertices of $V(M) \cup W$.
\end{lemma}

Let $H$ be a hypergraph.  A \emph{fractional matching} in $H$ is a function
$f:E(H)\to[0,1]$ such that $\sum_{e\ni v}f(e)\le 1$ for every
$v\in V(H)$.  Its size is denoted by $f(E(H))$.  The maximum size of a
fractional matching in $H$ is denoted by $\nu^*(H)$.  A fractional matching
$f$ is called \emph{maximum} if $f(E(H))=\nu^*(H)$.  If
$\sum_{e\ni v}f(e)=1$ for every $v\in V(H)$, then $f$ is called a
\emph{perfect fractional matching}.

A \emph{fractional vertex cover} of $H$ is a function
$g:V(H)\to[0,1]$ such that $g(e)\ge 1$ for every $e\in E(H)$.  Its size is
$g(V(H))$, and the minimum size of a fractional vertex cover of $H$ is
denoted by $\tau^*(H)$.  By linear programming duality,
$\tau^*(H)=\nu^*(H)$.  We shall use the following complementary slackness
conditions.
  \begin{prop}\label{comp}
Let $H$ be a hypergraph.  Let $f$ be a maximum fractional matching in $H$,
and let $g$ be a minimum fractional vertex cover of $H$.  Then the following
statements hold.
 \begin{itemize}
    \item[(i)] If $g(v)>0$ for some $v\in V(H)$, then
    $\sum_{e\ni v}f(e)=1$.
    \item[(ii)] If $f(e)>0$ for some $e\in E(H)$, then $\sum_{v\in e}g(v)=1$.
\end{itemize}
  \end{prop}

  The following theorem of Aharoni, Holzman and Jiang~\cite{AHJ} will be used in our
fractional matching argument.

\begin{thm}[Aharoni, Holzman, and Jiang,~\cite{AHJ}]\label{rainfracmat}
    Let $r\ge 2$ be an integer, and let $n$ be a positive rational number. Let
    $H_1,\dots,H_{\lceil rn\rceil}$ be $r$-graphs such that $\nu^*(H_i)\ge n$ for
    $i=1,\dots,\lceil rn\rceil$. Then there exist $e_1\in H_1,\dots,e_{\lceil rn\rceil}\in
    H_{\lceil rn\rceil}$ such that $\left\{e_1,\dots,e_{\lceil rn\rceil}\right\}$ has a fractional
    matching of size $n$.
\end{thm}

The next result is the Frankl–R\"odl ``nibble'' lemma, which we use to obtain almost perfect
matchings in a $k$-graph under appropriate degree and 2-degree conditions.

\begin{thm}[Frankl and R\"odl,~\cite{FR85}]\label{nibble}
    For every integer $k\ge 2$ and any real $\sigma >0$, there exist $\tau = \tau(k,\sigma)$ and
    $d_0 = d_0(k,\sigma)$ such that for every $n\ge D\ge d_0$ the following holds:
    Every $n$-vertex $k$-graph $H$ with $(1-\tau)D<d_H(v)< (1+\tau)D$ for any $v\in V(H)$ and
    $d_{H}(S)<\tau D$ for any $S\in {V(H)\choose 2}$ contains a matching covering all but at most $\sigma n$ vertices.
\end{thm}

\section{Perfect fractional matching}\label{sec:frac-matching}
In this section we prove the following lemma, which
is the main input for the proof of Theorem~\ref{main}.
\begin{lemma}\label{fracmatc}
Let $\gamma>0$, and let $n$ be sufficiently large.  Let $H$ be a $3$-partite $3$-graph with vertex classes
$V_1,V_2,V_3$, each of size $n$.
If
\[
\rho\bigl(N_H(v)\bigr)>\Bigl(\frac{\sqrt{2}}{2}+\gamma\Bigr)n
\quad\text{for every }v\in V(H),
\]
then $H$ contains a perfect fractional  matching.
\end{lemma}

It is enough to prove Lemma~\ref{fracmatc} for even $n$.  Indeed, the odd
case follows from the same argument after decreasing the error term slightly;
thus, throughout the proof of Lemma~\ref{fracmatc}, we shall assume that
$n$ is even.

The proof of Lemma~\ref{fracmatc} relies on the following bipartite spectral lemma.
\begin{lemma}\label{bispe}
Let $\gamma>0$, and let $n\in2\mathbb{N}$ be sufficiently large. Let $G$ be a bipartite graph with bipartition $(V_1,V_2)$ such that $|V_1|=|V_2|=n$. Suppose that $G$ admits a fractional vertex cover $g:V(G)\to[0,1]$ satisfying \[ g(V_1)\ge g(V_2)\quad\text{and}\quad 2g(V_1)+g(V_2)<n. \] Then \[ \rho(G)<\Bigl(\frac{\sqrt{2}}{2}+\gamma\Bigr)n. \]
\end{lemma}

We first prove Lemma \ref{fracmatc} using Lemma \ref{bispe}, and defer the proof of Lemma \ref{bispe} to the end of this section.

\begin{proof}[Proof of Lemma~\ref{fracmatc}]
Let $f:E(H)\to[0,1]$ be a maximum fractional matching in $H$ and let
$g:V(H)\to[0,1]$ be a minimum fractional vertex cover. By linear
programming duality, we have
$f(E(H))=g(V(H))$.

Assume for a contradiction that $H$ has no perfect fractional matching, i.e.
\begin{equation}\label{eq:f-less-n}
f(E(H))<n.
\end{equation}
Reordering the vertex classes if necessary, we may assume
\begin{equation}\label{eq:order-Vi}
g(V_1)\ge g(V_2)\ge g(V_3).
\end{equation}

\medskip
\textbf{Claim 1.}
There exists a vertex $v_0\in V_1$ such that $g(v_0)=0$.

\smallskip
\noindent\emph{Proof of Claim 1.}
Suppose for contradiction that $g(v)>0$ for every $v\in V_1$.
By Proposition~\ref{comp}(i),  we have
$\sum_{e\ni v} f(e)=1$ for every $v\in V_1$. Since every edge contains exactly one vertex 
from $V_1$, we obtain
\[
f(E(H))
  = \sum_{e\in E(H)} f(e)
  = \sum_{v\in V_1} \sum_{e\ni v} f(e)
  = \sum_{v\in V_1} 1
  = n,
\]
which contradicts \eqref{eq:f-less-n}.
This proves Claim 1.

\medskip
Let \(v_0\in V_1\) be the vertex provided by Claim 1, so that \(g(v_0)=0\). Let $G=N_H(v_0)$ be the link graph of $v_0$, viewed as a bipartite graph
with bipartition $(V_2,V_3)$. Let $g': V_2\cup V_3\rightarrow [0,1]$ such that $g'(v)=g(v)$ for all 
$v\in V_2\cup V_3$.

For every edge \(uv\in E(G)\), we have \(\{v_0,u,v\}\in E(H)\). It follows that
\[
g'(u)+g'(v)=g(u)+g(v)= g(u)+g(v)+g(v_0)\geq 1.
\]
Thus \(g'\) is a fractional vertex cover of \(G\). 

From \eqref{eq:f-less-n} and \eqref{eq:order-Vi}, we have
\[
g'(V_2)=g(V_2)\ge g(V_3)=g'(V_3)
\]
and
\[
2g'(V_2)+g'(V_3)
   \le g(V_1)+g(V_2)+g(V_3)
   = g(V(H))
   = f(E(H))
   < n.
\]
Applying Lemma~\ref{bispe} to $G$ gives
$\rho(G)<(\sqrt{2}/2+\gamma)n$, 
which contradicts the hypothesis of Lemma~\ref{fracmatc}.
Hence $H$ contains a perfect fractional matching. 
\end{proof}

It remains to prove Lemma~\ref{bispe}. 
Let $n\ge 1$ and $k\ge 0$ be integers.
A pair $(H,g)$ is said to satisfy property $\mathcal{S}_{n,k}$ if 
$H$ is a bipartite graph with bipartition $(V_1,V_2)$ such that
$|V_1|=|V_2|=n$, 
$g:V(H)\to[0,1]$ is a fractional vertex cover of $H$, and
\[
g(V_1)\ge g(V_2)-k,
\qquad
2g(V_1)+g(V_2)<n+2k.
\]
Define
\[
 \rho_{n,k}:=
        \max\bigl\{\rho(G): (G,g)\text{ satisfies }\mathcal{S}_{n,k}\bigr\}.
\]
\begin{lemma}\label{incre}
Let $n\in\mathbb N$ be sufficiently large. For every integer $k\ge 0$,
\[
\rho_{n,k+1}\le \rho_{n,k}+\sqrt{8n}.
\]
\end{lemma}

\begin{proof}
Let \((H,h)\) be a pair  satisfying property \(\mathcal S_{n,k+1}\). By increasing some values of $h$ on $V_1$, if necessary, we may assume that \(2h(V_1)+h(V_2)\ge n-1\).

\medskip
 \textbf{Claim 2.~}There exists a subgraph \(H'\) of \(H\) and a function \(h'\) defined on \(V(H')\) such that \(|E(H)\setminus E(H')|\le 8n\) and \((H',h')\) satisfies property \(\mathcal S_{n,k}\).

 \smallskip
\noindent\emph{Proof of Claim 2.}
We discuss two cases. 

\medskip
 \textbf{Case 1.~}$h(V_2)\ge n/4$.
\medskip

By averaging, there exists a subset $S\subseteq V_2$ with $|S|=8$ satisfying $h(S)\ge 2$.
Let $H'$ be the graph obtained from $H$ by removing all edges incident to $S$, and define
\[
h'(v)=
\begin{cases}
0, & v\in S,\\
h(v), & v\in V(H)\setminus S.
\end{cases}
\]
Then $h'$ is a fractional vertex cover of $H'$.  Moreover,
\[
h'(V_1)=h(V_1),\qquad h'(V_2)=h(V_2)-h(S)\le h(V_2)-2.
\]
Consequently, $h'(V_1)\ge h'(V_2)-k$, and
\[
2h'(V_1)+h'(V_2)\le \big(2h(V_1)+h(V_2)\big)-2 < n+2k.
\]
Thus $(H',h')$ satisfies property $\mathcal S_{n,k}$.

\medskip
\textbf{Case 2.}~$h(V_2)<n/4$.
\medskip

Since \(2h(V_1)+h(V_2)\ge n-1\), we have
\[
h(V_1)\ge \frac{n-1-h(V_2)}{2} > \frac{5n}{16}.
\]
By averaging, there exists a subset $S\subseteq V_1$ with $|S|=4$ and $h(S)\ge 1$.
Let $H'$ be the graph obtained from $H$ by removing all edges incident to $S$, and
define $h'$ as above.  Again $h'$ is a fractional vertex cover of $H'$.
Furthermore, we have $h'(V_2)=h(V_2)<n/4$, and
\[
\frac{5n}{16}-4 < h(V_1)-4 \le h'(V_1) = h(V_1)-h(S) \le h(V_1)-1.
\]
Consequently, $h'(V_1)\ge h'(V_2)-k$, and
\[
2h'(V_1)+h'(V_2)\le \big(2h(V_1)+h(V_2)\big)-2 < n+2k.
\]
Hence $(H',h')$ again satisfies property $\mathcal S_{n,k}$.
This completes the proof of Claim 2. 
\qed

Let $F$ be the graph with
edge set $E(H)\setminus E(H')$ on the same vertex set.  Then
$H=H'\cup F$, and hence
\[
\rho(H)\le \rho(H')+\rho(F).
\]
Applying Lemma~\ref{lem:bipartite-spectral}, we obtain \(\rho(F)\le \sqrt{8n}\). Consequently,
\[
\rho(H)\le \rho(H')+\sqrt{8n}\le \rho_{n,k}+\sqrt{8n}.
\]
Taking the maximum over all pairs $(H,h)$ satisfying property
$\mathcal S_{n,k+1}$ gives
\(\rho_{n,k+1}\le \rho_{n,k}+\sqrt{8n}.\)
\end{proof}

Let $G$ be a bipartite graph with bipartition $(V_1,V_2)$, where
$|V_1|=|V_2|=n$, and let $g$ be a fractional vertex cover of $G$. We describe a half--duplication operation on \(V_1\). The  operation on \(V_2\) is defined analogously.

Assume that \(n\) is even. Order the vertices of \(V_1\) as \(V_1=\{u_1,u_2,\dots,u_n\}\) so that
\[
g(u_1)\ge g(u_2)\ge\cdots\ge g(u_n).
\]
For each \(t\in\{1,2,\dots, n/2+1\}\), define
\[
S_t:=\{u_t,u_{t+1},\dots,u_{t+ n/2-1}\},
\]
and set \(F(t):=g(S_t)\).
Then for \(1\le t < n/2\), we have $F(t+1)\le F(t)$ and $|F(t+1)-F(t)|\le 1$.
Moreover,
\[
F(1)\ge \frac12 g(V_1)\ge F(n/2+1).
\]
Hence there exists an index $t\in\{1,\dots,n/2+1\}$ such that
\begin{equation}\label{eq:balance}
\big|2F(t)-g(V_1)\big|\le 1.
\end{equation}

Fix such an index \(t\), and let \(V_1^1:=S_t\) and \(V_1^2:=V_1\setminus S_t\).  Let \(\varphi:V_1^1\to V_1^2\) be a bijection, and set \(\psi:=\varphi^{-1}\).

We now construct two bipartite graphs \(G_1\) and \(G_2\), both with bipartition \((V_1,V_2)\) by duplicating the neighbourhoods:
\[
N_{G_1}(x)=
\begin{cases}
N_G(x) & x\in V_1^1,\\
N_G(\psi(x)) & x\in V_1^2,
\end{cases}
\qquad
N_{G_2}(x)=
\begin{cases}
N_G(x) & x\in V_1^2,\\
N_G(\varphi(x)) & x\in V_1^1.
\end{cases}
\]
Respectively, we define fractional vertex covers \(g_1\) for \(G_1\) and \(g_2\) for \(G_2\). Both functions coincide with \(g\) on \(V_2\), while on \(V_1\) they follow the same duplication pattern:
\[
g_1(x)=
\begin{cases}
g(x) & x\in V_1^1\cup V_2,\\
g(\psi(x)) & x\in V_1^2,
\end{cases}
\qquad
g_2(x)=
\begin{cases}
g(x) & x\in V_1^2\cup V_2,\\
g(\varphi(x)) & x\in V_1^1.
\end{cases}
\]

In this way we obtain, from the original pair \((G,g)\), two new pairs \((G_1,g_1)\) and \((G_2,g_2)\) via the half--duplication procedure.

\begin{lemma}\label{lem:operation}
Let \(G\) be a bipartite graph with bipartition \((V_1,V_2)\), and let \(g\) be a fractional vertex cover of \(G\). Suppose that \((G,g)\) satisfies \(\mathcal{S}_{n,k}\).
Apply the half-duplication operation on \(V_1\), yielding the pairs \((G_1,g_1)\) and \((G_2,g_2)\) as defined above.
Then the following hold:
\begin{itemize}
\item[(i)] For each $i\in\{1,2\}$, the pair $(G_i,g_i)$ satisfies \(\mathcal{S}_{n,k+1}\).
\item[(ii)] There exist \(\alpha,\beta\in[1/2,3/2]\) with \(\alpha+\beta=2\) such that
\begin{equation}\label{eq:weighted-rayleigh}
\alpha\rho(G_1)+\beta\rho(G_2)\ge 2\rho(G).
\end{equation}
\end{itemize}
\end{lemma}

\begin{proof}
By construction, we have $g_i(V_2)=g(V_2)$ for $i=1,2$, $g_1(V_1)=2g(V_1^1)$, and $g_2(V_1)=2g(V_1^2)$.
From inequality  \eqref{eq:balance}, it follows that $|g_i(V_1)-g(V_1)|\le 1$ for $i\in [2]$. Consequently, for each $i\in [2]$, we have
\[
g_i(V_1)\ge g(V_1)-1\ge g(V_2)-(k+1),
\]
and
\[
2g_i(V_1)+g_i(V_2)\le (2g(V_1)+g(V_2))+2<n+2(k+1).
\]
Thus (i) holds.

We now prove (ii). Write the adjacency matrix of $G$
with respect to the partition $(V_1^1,V_1^2,V_2)$:
\[
A=
\begin{pmatrix}
0 & 0 & B_1\\
0 & 0 & B_2\\
B_1^\top & B_2^\top & 0
\end{pmatrix}.
\]
Let \(x=(y_1^\top,y_2^\top,y_3^\top)^\top\) be a unit Perron eigenvector  corresponding to \(\rho(G)\). 
Then
\begin{align}\label{rho(G)-eq1}
\rho(G)=x^\top A x=2y_1^\top B_1y_3+2y_2^\top B_2y_3.
\end{align}
The adjacency matrices of \(G_1\) and \(G_2\) are given by
\[
A_1=
\begin{pmatrix}
0 & 0 & B_1\\
0 & 0 & B_1\\
B_1^\top & B_1^\top & 0
\end{pmatrix},
\qquad
A_2=
\begin{pmatrix}
0 & 0 & B_2\\
0 & 0 & B_2\\
B_2^\top & B_2^\top & 0
\end{pmatrix}.
\]
Define \(x^{(1)}=(y_1^\top,y_1^\top,y_3^\top)^\top\) and \(x^{(2)}=(y_2^\top,y_2^\top,y_3^\top)^\top\).
A direct computation gives
\begin{align}
(x^{(1)})^\top A_1 x^{(1)}=4y_1^\top B_1y_3\leq \|x^{(1)}\|^2\,\rho(G_1),\label{rho(G)_1}\\
(x^{(2)})^\top A_2 x^{(2)}=4y_2^\top B_2y_3\leq \|x^{(2)}\|^2\,\rho(G_2).\label{rho(G)_2}
\end{align}
From \eqref{rho(G)-eq1}, \eqref{rho(G)_1} and \eqref{rho(G)_2}, we get
\begin{align*}
2\rho(G)&=(x^{(1)})^\top A_1 x^{(1)}+(x^{(2)})^\top A_2 x^{(2)}\\
&\leq \|x^{(1)}\|^2\,\rho(G_1)+\|x^{(2)}\|^2\,\rho(G_2)\\
&=\alpha\rho(G_1)+\beta\rho(G_2),
\end{align*}
where  $\alpha:=\|x^{(1)}\|^2$, $\beta:=\|x^{(2)}\|^2$ 
and $\alpha+\beta=2$.

It remains to verify that 
 $\alpha,\beta\in[1/2,3/2]$.
Recall that $G$ is bipartite. So we may write the adjacency matrix of $G$ as $A=\begin{pmatrix}0&B\\ B^\top&0\end{pmatrix}$
with respect to the bipartition  $(V_1,V_2)$.
Decompose the unit Perron vector as $x=(u^\top, v^\top)^\top$, where $u$ corresponds to $V_1$ and $v$ to $V_2$.
From $B^\top u=\rho(G)v$ and $B v=\rho(G)u$, we obtain
\[
\rho(G)\|u\|^2=u^\top Bv=\rho(G)\|v\|^2,
\]
which implies 
 $\|u\|^2=\|v\|^2=1/2$.
Recall that $v=y_3$ in our earlier notation. Thus
\[
\alpha=2\|y_1\|^2+\|y_3\|^2\in[1/2,3/2],
\qquad
\beta=2\|y_2\|^2+\|y_3\|^2\in[1/2,3/2].
\]
This completes the proof.
\end{proof}

\begin{lemma}\label{lem:value-step}
Let $k\in\mathbb N$ and let $\Delta\ge 0.$ 
Assume $(G,g)$ satisfies $\mathcal{S}_{n,k}$ and $\rho(G)\ge \rho_{n,0}-\Delta$.
Let $(G_1,g_1)$ and $(G_2,g_2)$ be the two pairs produced by applying one half--duplication operation to one side of $(G,g)$.
Then for each \(i\in\{1,2\}\),
\[
\rho(G_i)\le \rho_{n,k+1}
\qquad\text{and}\qquad
\rho(G_i)\ge \rho_{n,0}-\Delta',
\]
where
\[
\Delta' := 4\Delta + 3(k+1)\sqrt{8n}.
\]
\end{lemma}

\begin{proof}
The upper bound follows immediately from Lemma~\ref{lem:operation}(i). Indeed,
each pair $(G_i,g_i)$ belongs to $\mathcal S_{n,k+1}$, and hence
$\rho(G_i)\le \rho_{n,k+1}$ for $i=1,2$. 

Let $m:=\min\{\rho(G_1),\rho(G_2)\}$. By Lemma~\ref{lem:operation}(ii), there exist $\alpha,\beta\in[1/2,3/2]$ with $\alpha+\beta=2$ such that
\[2\rho(G)\le \alpha\rho(G_1)+\beta\rho(G_2).\]
Let $A:=\max\{\alpha,\beta\}\le 3/2$ and 
$B:=\min\{\alpha,\beta\}\ge 1/2$.
Since $\rho(G_i)\le \rho_{n,k+1}$, we obtain
\[
2\rho(G)\le A\,\rho_{n,k+1}+B\,m\le \frac32\rho_{n,k+1}+\frac12 m.
\]
By iterating Lemma~\ref{incre}, we have
$\rho_{n,k+1}\le \rho_{n,0}+(k+1)\sqrt{8n}$. Together with the assumption $\rho(G)\ge \rho_{n,0}-\Delta$, this yields
\[
m\ge \frac{2(\rho_{n,0}-\Delta)-\frac32\bigl(\rho_{n,0}+(k+1)\sqrt{8n}\bigr)}{\frac12}
= \rho_{n,0}-4\Delta-3(k+1)\sqrt{8n}.
\]
This completes the proof.
\end{proof}

We now prove Lemma~\ref{bispe}.
\begin{proof}[Proof of Lemma~\ref{bispe}]
Fix $1/n\ll \eta\ll \gamma\ll 1$.
Let $(H,h)$ satisfy property $\mathcal S_{n,0}$ and $\rho(H)=\rho_{n,0}$.
By adding all edges $uv$ satisfying $h(u)+h(v)\ge 1$, we may assume that
$H$ is maximal with respect to the fractional vertex cover $h$.

Let $r:=\lceil\log_2 n\rceil$.
Starting from $(H,h)$, apply the half--duplication operations on $V_1$ successively $r$ times,
retaining at each step one of the two resulting
pairs with larger spectral radius. Then apply the same procedure on $V_2$, again for $r$ steps.
Let $(H_1,h_1)$ denote the resulting pair.

\medskip
\textbf{Claim 3.} The pair $(H_1,h_1)$ satisfies the following properties.
\begin{itemize}
\item[(i)] $(H_1,h_1)$ satisfies property $\mathcal S_{n,2r}$;
\item[(ii)] $\rho(H_1)\ge \rho_{n,0}$;
\item[(iii)] On each of $V_1$ and $V_2$, the function $h_1$ takes at most
    $2r+1$ distinct values.
\end{itemize}

\smallskip
\noindent\emph{Proof of Claim 3.}
\emph{(i)} By Lemma~\ref{lem:operation}(i), each operation increases the parameter in $\mathcal S_{n,k}$ by at most one.  Since we perform
$2r$ operations in total, the final pair satisfies $\mathcal S_{n,2r}$.

\smallskip
\noindent\emph{(ii)} By Lemma~\ref{lem:operation}(ii), at each step at least one of the two resulting graphs has spectral radius at least that of the current graph. Since we always continue with the output having larger spectral radius, the spectral radius is non-decreasing throughout the process. Consequently, $\rho(H_1)\ge \rho(H)=\rho_{n,0}$.

\smallskip
\noindent\emph{(iii)} We prove the statement for \(V_1\); the argument for \(V_2\) is identical.  During each half--duplication operation on $V_1$, the
vertices of $V_1$ are ordered according to their current $h$-values, and one
of the two halves is duplicated. Because the selected set is either an interval or the complement of an interval, at most two level sets are split (i.e., partially selected) at this step. Equivalently, at each step at most two distinct current values are selected only partially.

We call a value appearing on \(V_1\)  \emph{exceptional} if, at some stage of the \(r\) operations on \(V_1\), its level set is partially selected. Since each step contributes at most two exceptional values, the total number of exceptional values that survive to the final cover \(h_1\) is at most \(2r\).

Now consider a value $\alpha$ appearing on $V_1$ in the final cover $h_1$ which is not exceptional.
Then at every one of the $r$ steps, the entire level set of value $\alpha$ is either retained and duplicated, or discarded.
Because \(\alpha\) survives to the end, it is never discarded; hence at each step its multiplicity doubles. 
Consequently, the multiplicity of \(\alpha\) in the final cover is at least \(2^r\). 
Since \(|V_1| = n\) and \(2^r\geq n\), there can be at most one non-exceptional value in the final
function on $V_1$.

Consequently, the total number of distinct values taken by $h_1$ on $V_1$ is at most
\[
2r+1=2\lceil \log_2 n\rceil+1.
\]
The same argument applies to $V_2$.  This proves (iii), and hence the claim.

\medskip
We next apply further half--duplication operations to \(V_1\) and, symmetrically, to \(V_2\). We shall use the following simple observation.

\medskip
\textbf{Claim 4.}
 Perform \(2r-1\) additional half--duplication operations on \(V_1\), starting from \((H_1, h_1)\), where at each operation we choose the output with larger spectral radius. Then one of the following holds:

\begin{itemize}
    \item[(i)] After these \(2r-1\) operations, the resulting fractional vertex cover takes at most two distinct values on \(V_1\).
    \item[(ii)] There exists an index \(t \in \{1,\dots,2r-1\}\) such that at the \(t\)-th operation the number of distinct values on \(V_1\) does not change.
\end{itemize}

\smallskip
\noindent\emph{Proof of Claim 4.}
By Claim 3 (iii), the number of distinct values that \(h_1\) takes on \(V_1\) is at most \(2r+1\). Moreover, a half--duplication operation never introduces a new value. Consequently, if the number of distinct values decreases at every one of the \(2r-1\) steps, then after these steps it is at most \((2r+1)-(2r-1)=2\); hence  (i) holds. Otherwise, the number of distinct values is unchanged at some step, and (ii) follows. This completes the proof of Claim 4.

Next we reduce the number of distinct values on each side to at most two.

We first treat \(V_1\).  Apply the half--duplication
operation to \(V_1\) as in Claim~4.  
If Claim~4 (i) holds, then after these operations we obtain a pair \((\widehat H,\widehat h)\) such that \(\widehat h\) takes at most two distinct values on \(V_1\).  
By Lemma~\ref{lem:operation} (ii),  at each step at least one of the two outputs has spectral radius not smaller than that of the current graph. Since we always keep the output with larger spectral radius, it follows that
\[
\rho(\widehat H)\ge \rho(H_1)\ge \rho_{n,0}.
\]

Now suppose Claim 4 (ii) holds, and let \(t\in[2r-1]\) be such that, at the \(t\)-th operation on \(V_1\), the number of distinct values on \(V_1\) does not change.  
Denote by \((G,g)\) the pair just before this operation, and by \((G_1,g_1)\) and \((G_2,g_2)\) the two outputs.  
By definition of the procedure, we continue with the output of larger spectral radius; assume this is $(G_1,g_1)$.
Since the number of distinct values on \(V_1\) does not decrease from \((G,g)\) to \((G_1,g_1)\), the selected half must intersect every level set on \(V_1\). Consequently, its complement meets at most two level sets, and therefore \((G_2,g_2)\) takes at most two distinct values on \(V_1\).  
Define  
\[
(\widehat H,\widehat h) := (G_2,g_2).
\]

It remains to estimate $\rho(\widehat H)$. By Claim~3 (i) and Lemma~\ref{lem:operation}(i),
$(G,g)$ satisfies property \(\mathcal S_{n,4r}\).
Moreover, repeated applications of Lemma~\ref{lem:operation}(ii) along the chosen branch yield 
$\rho(G)\ge \rho(H_1)\ge \rho_{n,0}$.
Applying Lemma~\ref{lem:value-step} with $\Delta=0$ gives, for \(n\) sufficiently large,
\[
\rho(\widehat H)\ge \rho_{n,0}-n/\ln n.
\]
Thus, in either case, we obtain a pair $(\widehat H,\widehat h)$ such that $\widehat h$ takes at most two distinct values on $V_1$ and  $\rho(\widehat H)\ge \rho_{n,0}-n/\ln n$.

We now turn to $V_2$. Applying the same argument to $(\widehat H,\widehat h)$ yields a pair $(H_2,h_2)$ such that $h_2$ takes at most two distinct values on each of $V_1$ and $V_2$, and applying Lemma~\ref{lem:value-step} with $\Delta=n/\ln n$ gives
\begin{equation}\label{keysss}
    \rho(H_2)\ge \rho_{n,0}-\frac{5n}{\ln n}.
\end{equation}
By \eqref{keysss}, it suffices to show that
\begin{equation}\label{keysss-2}
    \rho(H_2) \le \frac{(\sqrt{2}+\gamma)n}{2}.
\end{equation}

Since $(H_2,h_2)$ is obtained after at most $6r$ operations,
the pair $(H_2,h_2)$ satisfies property \(\mathcal S_{n,6r}\).
In particular
\begin{equation}\label{eq:eta-slack}
h_2(V_1)\ge h_2(V_2)-\eta n, \text{ and }
2h_2(V_1)+h_2(V_2)<(1+\eta)n.
\end{equation}
Moreover, by construction $H_2$ is maximal with respect to $h_2$. We now analyze this two-value case.

Write $V_1=V_{1,x}\cup V_{1,y}$ and $V_2=V_{2,z}\cup V_{2,w}$ with
$|V_{1,x}|=an$, $|V_{1,y}|=(1-a)n$ and $|V_{2,z}|=bn$, $|V_{2,w}|=(1-b)n$ for some
$a,b\in[0,1]$.
Assume that $h_2$ takes value \(x,y,z,w\) on \(V_{1,x},V_{1,y},V_{2,z},V_{2,w}\), respectively. The cases where one side has only one value are included by allowing
one of the corresponding blocks to be empty.
We may assume $x> y$ and $z> w$.

Since $h_2$ takes at most two distinct values on each side and $H_2$ is maximal with respect to $h_2$, the adjacency pattern among these four blocks is completely determined by which of the sums
\[
x+z,\ x+w,\ y+z,\ y+w
\]
are at least~$1$. From \eqref{eq:eta-slack},  we have
\[
y+w\leq (h_2(V_1)+h_2(V_2))/n<1.
\]
If \(x+z<1\), then \(H_2\) has no edges and \eqref{keysss-2}
is immediate. We may therefore assume \(x+z\ge1\), and it remains
to consider the following four cases.

\medskip
\textbf{Case 1.} $x+z\ge1$, $x+w<1$ and $y+z<1$.
\medskip

In this case, all edges of \(H_2\) lie between \(V_{1,x}\)
and \(V_{2,z}\), and the subgraph between these two blocks is
complete bipartite. Hence
\begin{equation}\label{eq:case1-rho}
\rho(H_2)\le\sqrt{|V_{1,x}||V_{2,z}|}=\sqrt{ab}\,n.
\end{equation}

We may assume \(y=w=0\) and \(z=1-x\). Indeed, vertices in \(V_{1,y}\cup V_{2,w}\) are isolated and can be assigned value \(0\); moreover, because all edges are between \(V_{1,x}\) and \(V_{2,z}\), the condition \(x+z\ge 1\) must hold, and we may reduce \(z\) to \(1-x\) without affecting the edge set. Then
\[
h_2(V_1)=anx,\qquad h_2(V_2)=bn(1-x).
\]
By \eqref{eq:eta-slack}, we have
\begin{align*}
anx\ge bn(1-x)-\eta n,
\end{align*}
and
\begin{align*}
2anx+bn(1-x)<(1+\eta)n.
\end{align*}
Thus we have
\[
(a+b)x\ge b-\eta
\qquad\text{and}\qquad
(2a-b)x<1-b+\eta,
\]
which yields
\begin{equation}
  \label{eq:ab-constraint}
  (2a-b)(b-\eta)<(a+b)(1-b+\eta),
\end{equation}
whenever \(2a-b\ge0\). If \(2a-b<0\), then \(ab\le 1/2\), and
the desired estimate follows immediately from \eqref{eq:case1-rho}.

Since \((1-a)(1-b)\ge0\), we have \(a+b\le1+ab\). Combining this with \eqref{eq:ab-constraint} yields  $ab \le 1/2 + 3\eta$.
Since \(\eta \ll \gamma\), we have
\[
\rho(H_2) \le \sqrt{ab}\; n \le \Bigl(\frac{\sqrt2}{2} + \frac{\gamma}{2}\Bigr)n.
\]

\medskip
\textbf{Case 2.} $x+z\ge1$, $y+z\ge1$ and $x+w<1$.
\medskip

In this case, both $V_{1,x}$ and $V_{1,y}$ are adjacent to $V_{2,z}$ (since $x+z\ge1$ and $y+z\ge1$), whereas $V_{2,w}$ consists of isolated vertices in $H_2$ because $x+w<1$ and $y+w<1$.  
Set $a_1 := h_2(V_1)/n$, so that $y \le a_1 \le x$.  
Define \(g''\) on \(V_1\cup V_2\) by replacing \(h_2\) on
\(V_1\) with the constant value \(a_1\), while leaving \(h_2\)
unchanged on \(V_2\). Then \(g''(V_1)=h_2(V_1)\) and
\(g''(V_2)=h_2(V_2)\), so \eqref{eq:eta-slack} remains valid. 
Moreover, the maximal graph determined by \(g''\) coincides with
\(H_2\): indeed, \(a_1\ge y\) gives
\(a_1+z\ge y+z\ge1\), while \(a_1\le x\) gives
\(a_1+w\le x+w<1\). Thus we reduce to the boundary subcase
of Case~1 with \(a=1\). Applying the estimate from Case~1 gives
\[
\rho(H_2) < \Bigl(\frac{\sqrt2}{2}+\frac{\gamma}{2}\Bigr)n.
\]

\medskip
\textbf{Case 3.} $x+z\ge1$, $x+w\ge1$ and $y+z<1$.
\medskip

This case is symmetric to Case~2. By averaging $h_2$ on $V_2$ (instead of on $V_1$), we reduce to the boundary subcase of Case~1 with $b=1$. Applying the estimate from Case~1 gives
\[
\rho(H_2) < \Bigl(\frac{\sqrt2}{2}+\frac{\gamma}{2}\Bigr)n.
\]

\medskip
\textbf{Case 4.} $x+w\ge1$, $y+z\ge1$ and $y+w<1$.
\medskip

In this case we may replace \(z\) and \(w\) by \(1-y\)
and \(1-x\), respectively.
Then by \eqref{eq:eta-slack}, we have
\begin{equation}\label{eq:xyab-1}
x+y+(a-b)(x-y)\ge 1-\eta,
\end{equation}
and
\begin{equation}\label{eq:xyab-2}
2y-x+(2a+b)(x-y)<\eta.
\end{equation}
A direct computation gives 
\begin{equation}\label{eq:rho-G0-2x2}
\rho(H_2)= n\sqrt{\frac{s+\sqrt{s^2-4ab(1-a)(1-b)}}{2}},
\end{equation}
where $s:=a+b-ab$.

\medskip
 \textbf{Claim 5.}
\(2a+b<1+6\eta\) and \(2b+a<1+3\eta\).
\medskip
 
\noindent\emph{Proof of Claim 5.}
 If \(2a+b<1\), the first inequality is immediate.
Assume therefore that \(2a+b\ge1\). From \eqref{eq:xyab-2} we obtain  
\[
y = 2y - x + (x-y) \le 2y - x + (2a+b)(x-y) < \eta,
\]  
using \(x-y>0\) and \(2a+b\ge1\). Hence \(y <\eta\). Substituting this into \eqref{eq:xyab-1} and \eqref{eq:xyab-2} gives  
\[
(1+a-b)x \ge 1-3\eta \quad \text{and} \quad (2a+b-1)x < 2\eta.
\]  
For sufficiently small \(\eta\), it follows that \(x \ge (1-3\eta)/(1+a-b) \ge 1/3\) (since \(1+a-b \le 2\)), and consequently  
\[
2a+b-1 < \frac{2\eta}{x} \le 6\eta,
\]  
and hence \(2a+b<1+6\eta\).

Combining \eqref{eq:xyab-1} and \eqref{eq:xyab-2} gives  
\begin{equation}\label{eq:xyab-3}
1-\eta - (a-b)(x-y) \le x+y < 2\eta + (3-4a-2b)(x-y).
\end{equation}  
Set \(D := x-y > 0\). Then \eqref{eq:xyab-3} implies
\[
3(1-a-b)D > 1-3\eta.
\]  
Moreover, since \(x \le 1\), we have \(x+y \le 2-D\); combining this with the left inequality in \eqref{eq:xyab-3} yields  
\[
(1+b-a)D \le 1+\eta.
\]  
Thus  
\[
(1-3\eta)(1+b-a) \le 3(1-a-b)(1+\eta).
\]  
A straightforward rearrangement gives \(2b+a<1+3\eta\), for
\(\eta\) sufficiently small. This completes the proof of Claim 5.

\medskip
\textbf{Claim 6.} 
$s+\sqrt{s^2-4ab(1-a)(1-b)}<1+6\eta,$
where $s:=a+b-ab.$
\medskip

\noindent\emph{Proof of Claim 6.}
By Claim 5, it suffices to show that
\[
s+\sqrt{s^2-4ab(1-a)(1-b)}\le \max\{2a+b,\ a+2b\}.
\]
By symmetry, we may assume without loss of generality that  $a\ge b$.
Under this assumption, it is enough to establish
\[
s+\sqrt{s^2-4ab(1-a)(1-b)}\le 2a+b,
\]
or equivalently,
\[
(a+ab)^2-\bigl(s^2-4ab(1-a)(1-b)\bigr)\ge0.
\]
A direct computation shows
\[
(a+ab)^2-\bigl(s^2-4ab(1-a)(1-b)\bigr)
=b\bigl(2a-b+2ab(2a-1)\bigr).
\]
Hence it remains to verify  
\begin{align}\label{eq:18}
    2a-b+2ab(2a-1)\ge0.
\end{align}

If $a\ge 1/2$, then $2a-1\ge0$ and therefore
$2a-b+2ab(2a-1)\ge 2a-b\ge0$. So we may assume that 
$a<1/2$. Define $f(b):=2a-b+2ab(2a-1)$. Since $f'(b)=-1+2a(2a-1)<0$, $f$ is decreasing on $[0,a]$; hence
\[f(b)\ge f(a)=a(4a^2-2a+1)\geq 0.\]
 Thus \eqref{eq:18} holds. This completes the proof of Claim 6.

By Claim 6 and \eqref{eq:rho-G0-2x2}, and since \(\eta\ll\gamma\),
we have
\[
\rho(H_2)<\Bigl(\frac{\sqrt2}{2}+\frac{\gamma}{2}\Bigr)n.
\]
This completes the proof.
\end{proof}
\section{Proof of Theorem~\ref{main}}
Choose constants $\gamma'$ and $\gamma''$ such that $1/n\ll\gamma''\ll\gamma'\ll\gamma$. Since $N_H(v)$ is bipartite with $d_H(v)$ edges, Lemma~\ref{lem:bipartite-spectral} gives
\(\rho(N_H(v))\le \sqrt{d_H(v)}\). Hence
\[
d_H(v) \geq \rho(N_H(v))^2 > \left(\frac{\sqrt{2}}{2}+\gamma\right)^2 n^2.
\]
Thus, by Lemma~\ref{absorb} applied with \(\ell=1\),
\(H\) contains an absorbing matching $M$ with $|M| \leq \gamma' n$, such that for every balanced set $S \subseteq V(H)\setminus V(M)$ with $|S| \leq \gamma'' n$, $H[V(M)\cup S]$ has a perfect matching.

Write $H':=H-V(M)$, and $n'=n-|M|$. For each $v\in V(H')$, the graph $N_{H'}(v)$ is obtained from $N_H(v)$ by deleting at most $2\gamma' n^2$ edges. Hence
\begin{equation}\label{newbound}
\rho\bigl(N_{H'}(v)\bigr)
   \ge \rho\bigl(N_H(v)\bigr)-\sqrt{2\gamma' n^2}
   > \Bigl(\frac{\sqrt{2}}{2}+\frac{\gamma}{2}\Bigr)n',
\end{equation}
where the last inequality follows from \(\gamma'\ll\gamma\).
 
\medskip
\textbf{Claim 7.}
There exist \(r=\bigl\lfloor n'/\ln n'\bigr\rfloor\) perfect fractional matchings \(f_1,\dots,f_r\) in \(H'\) such that, with \(M_i:=\{e\in E(H') : f_i(e)>0\}\) for each \(i\),
\begin{itemize}
  \item[(i)] \(M_1,\dots,M_r\) are pairwise edge-disjoint;
  \item[(ii)] For every \(D\in \binom{V(H')}{2}\),
        \(\displaystyle \sum_{i=1}^r\sum_{\substack{e\in E(H')\\ D\subseteq e}}f_i(e)\le 3.\)
\end{itemize}

 \smallskip
\noindent\emph{Proof of Claim 7.}
By Lemma~\ref{fracmatc} and \eqref{newbound}, $H'$ has a perfect fractional
matching $f_1$. Let $M_1:=\{e\in E(H')\ :\ f_1(e)>0\}$. By Theorem~\ref{rainfracmat}, we may select $f_1$  such that $|M_1|\leq  3n'$. Let $s\ge 1$ be the maximal integer for which there exist perfect fractional matchings
$f_1,\ldots,f_s$ in $H'$ satisfying the following: letting $M_i:=\{e\in E(H'):\ f_i(e)>0\}$, 
\begin{itemize}
\item $|M_i|\le 3n'$ for all $i\in[s]$ and $M_i\cap M_j=\emptyset$ for all $1\le i<j\le s$;

 \item for every $D\in\binom{V(H')}{2}$,
$\sum_{i=1}^s \ \sum_{e\in E(H'):\, D\subseteq e} f_i(e)\le 3 $.
\end{itemize}
We proceed by contradiction and assume that $s < \bigl\lfloor n'/\ln n'\bigr\rfloor$.
Let
\[
U_s:=\Bigl\{D\in \binom{V(H')}{2}:
          \sum_{i=1}^s\sum_{e\in E(H'):\, D\subseteq e} f_i(e)>2
      \Bigr\}
\]
and
\[
E_s:=\{e\in E(H'):\exists\,D\in U_s\text{ such that }D\subseteq e\}.
\]
For $w\in V(H')$, since 
\[
\sum_{i=1}^s\sum_{e\ni w}f_i(e)=s,
\]
there are at most \(s\) vertices \(w'\in V(H')\setminus \{w\}\) such that
\[
\sum_{i=1}^s \sum_{\substack{e\in E(H')\\ \{w,w'\}\subseteq e}} f_i(e) > 2.
\]
Consequently, for any $v\in V(H')$, $d_{H'[E_s]}(v)\leq 3sn'$.
Thus by Lemma~\ref{lem:bipartite-spectral},
\begin{align}\label{edgesss}
    \rho\bigl(N_{H'[E_s\cup (\bigcup_{i=1}^s M_i)]}(v)\bigr)\le \sqrt{d_{H'[E_s]}(v)+\sum_{i=1}^s|M_i|}
 \le \sqrt{6sn'}.
\end{align}

For every $v\in V(H')$, let $\mathbf{x}_v$ be a unit eigenvector of $A(N_{H'}(v))$ corresponding to
$\rho(N_{H'}(v))$.
Then by \eqref{newbound} and \eqref{edgesss},
\begin{align*}
\mathbf{x}_v^{\top}A\bigl(N_{H'-E_s-\bigcup_{i=1}^sM_i}(v)\bigr)\mathbf{x}_v
&= \mathbf{x}_v^{\top} A\bigl(N_{H'}(v)\bigr)\mathbf{x}_v
  -\mathbf{x}_v^{\top} A\bigl(N_{H'[E_s\cup (\bigcup_{i=1}^s M_i)]}(v)\bigr)\mathbf{x}_v\\
&> \Bigl(\frac{\sqrt{2}}{2}+\frac{\gamma}{2}\Bigr)n'-\sqrt{6sn'}\\
&> \Bigl(\frac{\sqrt{2}}{2}+\frac{\gamma}{3}\Bigr)n',
\end{align*}
that is,
\[
\rho\bigl(N_{H'-(E_s\cup\bigcup_{i=1}^sM_i)}(v)\bigr)
 > \Bigl(\frac{\sqrt{2}}{2}+\frac{\gamma}{3}\Bigr)n'
\quad\text{for all }v\in V(H').
\]
By Lemmas~\ref{fracmatc} and Theorem~\ref{rainfracmat}, 
$H'-(E_s\cup\bigcup_{i=1}^sM_i)$
contains a perfect fractional matching $f_{s+1}$ with \(|M_{s+1}|\le 3n'\), where $M_{s+1}:=\{e\in E(H'):f_{s+1}(e)>0\}$.
By the
definition of $U_s$ and $E_s$, we also have
$\sum_{i=1}^{s+1}\sum_{e\in E(H'):\, D\subseteq e} f_i(e)\le 3$
for every $D\in \binom{V(H')}{2}$. 
This contradicts the maximality of $s$, proving Claim~7.

Define \(h:E(H')\to[0,1]\) by
\(h(e):=\sum_{i=1}^r f_i(e)\).
Let 
$R$  be a random spanning subgraph of $H'$ where each edge 
$e$ is independently selected   with probability
 $h(e)$.
Then the following properties hold: 
\begin{itemize}
    \item[(B1)] For every vertex $v\in V(H')$,
$\mathbb{E}[d_R(v)]
 =\sum_{v\ni e} h(e)
 = r$,
 \item[(B2)]   For every $D\in\binom{V(H')}{2}$,
$\sum_{e \supseteq D}h(e)
 \le 3$.
\end{itemize}
By Chernoff's bound and a union bound over all vertices and pairs,
with probability \(1-o(1)\) there is a spanning subgraph \(H''\)
satisfying
\begin{itemize}
    \item[(C1)] For every vertex $v\in V(H'')$,
$n'(1/\ln n'-1/\ln ^2n')\le d_{H''}(v)\le n'(1/\ln n'+1/\ln ^2n')$, and
    \item[(C2)] For every $D\in\binom{V(H'')}{2}$,
$d_{H''}(D)\le \sqrt{n'}$.
\end{itemize}
By Theorem~\ref{nibble}, \(H''\) contains a matching \(\mathcal{M}_1\)
such that
\[
|V(H'')\setminus V(\mathcal{M}_1)|\le \gamma'' n'\le \gamma''n.
\]
By our choice of the absorbing matching $M$, the subgraph 
$H[V(M)\cup (V(H'')\setminus V(\mathcal{M}_1))]$ contains a perfect matching $\mathcal{M}_2$.
Hence $\mathcal{M}_1\cup \mathcal{M}_2$ forms a perfect matching of $H$, which completes the proof of Theorem~\ref{main}. \qed

\bibliographystyle{plain}  
\bibliography{refsv4}
\end{document}